\documentclass[12pt,twoside,a4paper,reqno]{amsart}
\usepackage{amsmath,amsthm,amssymb,amscd,amsfonts}
\usepackage{enumerate}
\makeatletter
\@namedef{subjclassname@2010}{%
  \textup{2010} Mathematics Subject Classification}
\makeatother
\theoremstyle{definition}

\newtheorem{theo}{Theorem}[section]

\newtheorem{prop}[theo]{Proposition}
\theoremstyle{definition}

\theoremstyle{definition}

\newtheorem{rem}[theo]{Remark}
\numberwithin{equation}{section}

\title{ pointwise amenable, non-amenable Banach algebras }

\author{Sara Behnamian}
\address{Department of Mathematics, Science and research Branch, Islamic Azad University, Tehran, Iran, e-mail: {\tt sara.behnamian@srbiau.ac.ir}}

\author{Amin Mahmoodi}
\address{Department of Mathematics, Central Tehran Branch, Islamic Azad University, Tehran, Iran, e-mail: {\tt a\_mahmoodi@iauctb.ac.ir}}

\begin{document}
\pagestyle{headings}

\begin{abstract}
It is shown that a pointwise amenable Banach algebra need not be
amenable. This positively answer a question raised by Dales,
Ghahramani and Loy.
\end{abstract}

\maketitle Keywords: amenable, pointwise amenable, Connes amenable,
character amenable.

\section{Introduction}
Let $\mathfrak{A}$ be a Banach algebra, and let $E$ be a Banach
$\mathfrak{A}$-bimodule. A \textit{derivation} is a bounded linear
map $D: \mathfrak{A} \longrightarrow E$
 satisfying $D(ab)=Da\cdot b+a\cdot Db\;\; (a,b\in \mathcal{A})$.  A Banach algebra $\mathfrak{A}$ is
 \textit{amenable} if for any Banach
$\mathfrak{A}$-bimodule $E$, any derivation $D: \mathfrak{A}
\longrightarrow E^*$ is \textit{inner}, that is, there exists $ \eta
\in E^*$ with $ D(a) = ad_{\eta}(a) = a \cdot \eta - \eta \cdot a$,
$a \in \mathfrak{A}$. This powerful notion introduced by B. E.
Johnson \cite {B. E}. The pointwise variant of amenability
introduced by H. G. Dales,  F. Ghahramani and R. J. Loy, and
appeared formally in \cite{DL}.  A Banach algebra  $\mathfrak{A}$ is
\textit{pointwise  amenable at} $a \in \mathfrak{A}$ if for any
Banach $\mathfrak{A}$-bimodule $E$, any derivation $D: \mathfrak{A}
\longrightarrow E^*$ is \textit{pointwise inner at} $a $, that is,
there exists $ \eta \in E^*$ with $ D(a ) = ad_{\eta}(a ) $.
Further, $\mathfrak{A}$ is \textit{pointwise  amenable} if it is
pointwise amenable at each $a \in \mathfrak{A}$.

For a Banach algebra $\mathfrak{A}$, recall that the projective
tensor product $\mathfrak{A} \widehat{\otimes} \mathfrak{A}$ is a
Banach $\mathfrak{A}$-bimodule in a natural way and the bounded
linear map $ \pi : \mathfrak{A} \widehat{\otimes} \mathfrak{A}
\longrightarrow \mathfrak{A}$ defined by $\pi( a \otimes b)  = ab $,
$(a, b \in \mathfrak{A})$ is a Banach $\mathfrak{A}$-bimodule
homomorphism.

Obviously, every amenable Banach algebra is pointwise  amenable. For
the converse, however, it has been open so far if there is a
pointwise amenable Banach algebra which is not already amenable
\cite[Problem 5]{DL}.

In this note, we give an illuminating example to show that the
problem has an affirmative answer.

\section{The results}

Let $\mathcal{V}$ be a Banach space, and let $f \in \mathcal{V}^*$
(the dual space of $\mathcal{V}$) be a non-zero element such that
$|| f || \leq 1$. Then $\mathcal{V}$ equipped with the product
defined by $ a b := f(a) b$ for $a , b \in\mathcal{V}$, is a Banach
algebra denoted by $\mathcal{V}_{f} $. In general, $\mathcal{V}_{f}
$ is a non-commutative, non-unital Banach algebra without right
approximate identity. One may see \cite{des, kv} for more details
and properties on this type of algebras.

A small variation of standard argument in \cite{J} gives the
following characterization of  pointwise amenability.
\begin{prop}\label{1.1} A Banach algebra $\mathfrak{A}$ is pointwise
amenable if and only if for each $a \in \mathfrak{A}$ there exists a
net $ (m_\alpha)_{\alpha} \subseteq \mathfrak{A} \widehat{\otimes}
\mathfrak{A}$ (depending on $a$) such that $ a \cdot m_\alpha -
m_\alpha \cdot a \longrightarrow 0$ and $ a \pi(m_\alpha)
\longrightarrow a$.
\end{prop}
\begin{theo}\label{1.2} Let $\mathcal{V}$ be a Banach space, and let $f \in
\mathcal{V}^*$ be a non-zero element such that $|| f || \leq 1$.
Then $\mathcal{V}_{f} $ is pointwise amenable, but not amenable.
\end{theo}
\begin{proof} For $a \in \mathcal{V}_{f}$, we consider $ m:= \sum_{n=1}^\infty b_n \otimes c_n \in \mathcal{V}_{f} \widehat{\otimes}
\mathcal{V}_{f}$ with $\sum_{n=1}^\infty f(b_n) = \frac{1}{f(a)} $
and $c_n := a$ for all $n$. Then $$ a \cdot m = \sum_{n=1}^\infty a
b_n \otimes c_n = f(a) \sum_{n=1}^\infty  b_n \otimes c_n = f(a) m
$$ and similarly $m \cdot a = f(a) m$, so that $  a \cdot m = m \cdot
a$. Next, $\pi(m) = \sum_{n=1}^\infty b_n c_n= (\sum_{n=1}^\infty
f(b_n)) a = \frac{1}{f(a)} a$. Hence $$ a \pi(m) = \frac{1}{f(a)}
a^2 = \frac{1}{f(a)} f(a) a = a \ .$$ Therefore, $\mathcal{V}_{f} $
is pointwise amenable by Proposition \ref{1.1}. Because of the lack
of a bounded approximate identity, $\mathcal{V}_{f} $ is not
amenable.
\end{proof}
\begin{rem}\label{1.3} Similar to \cite{DL}, we may define pointwise contractible Banach algebras
(see for instance \cite{D} for the definition of contractible Banach
algebras). Then Theorem \ref{1.2} says that $\mathcal{V}_{f} $ is
pointwise contractible. Notice that $\mathcal{V}_{f} $ is not
contractible, because it has no identity.
\end{rem}

Let $\mathfrak{A}$ be a Banach algebra. A Banach
$\mathfrak{A}$-bimodule $E$ is \textit{dual} if there is a closed
submodule $E_*$ of $E^*$ such that $E=(E_*)^*$. A Banach algebra
$\mathfrak{A}$ is \textit{dual} if it is a dual Banach space such
that multiplication is separately continuous in the $w^*$-topology.
 For a dual Banach algebra $\mathfrak{A}$, a dual Banach $\mathfrak{A}$-bimodule $E$ is \textit{normal} if the module actions of
  $\mathfrak{A}$ on $E$ are
 $w^*$-continuous. The notion of Connes amenability for dual Banach algebras, which is
another modification of the notion of amenability
  systematically introduced   by  V. Runde  \cite{R1}.
 A dual Banach algebra $\mathfrak{A}$ is \textit{Connes
amenable} if every $w^*$-continuous derivation from $\mathfrak{A}$
into a normal, dual Banach
 $\mathfrak{A}$-bimodule is inner. The pointwise variant of Connes amenability introduced in \cite{sm}.
  A dual Banach algebra $\mathfrak{A}$ is \textit{pointwise Connes amenable at } $a \in
\mathfrak{A}$ if for every normal, dual Banach
$\mathfrak{A}$-bimodule $E$, every $w^*$-continuous derivation $D:
\mathfrak{A}\to E$ is pointwise inner at $a$. Moreover,
$\mathfrak{A}$ is \textit{pointwise Connes amenable} if it is
pointwise Connes amenable at each $a \in \mathfrak{A}$.

\begin{theo}\label{1.8} Let $\mathcal{V}$ be a dual Banach space, and let $f \in
\mathcal{V}^*$ be a $w^*$-continuous non-zero element such that $||
f || \leq 1$. Then $\mathcal{V}_{f} $ is a pointwise Connes
amenable, non-Connes amenable dual Banach algebra.
\end{theo}
\begin{proof}Since $f$ is $w^*$-continuous, the multiplication on $\mathcal{V}_{f}
$ is separately  $w^*$-continuous, and so $\mathcal{V}_{f} $ is a
dual Banach algebra. As $\mathcal{V}_{f} $ is pointwise amenable by
Theorem \ref{1.2}, it is automatically pointwise Connes amenable.
Since $\mathcal{V}_{f}$ does not have an identity, it is not Connes
amenable by \cite[Proposition 4.1]{R1}.
\end{proof}

The following is \cite[Theorem 2.6]{sm}.
\begin{theo}\label{1.11}
Every commutative, pointwise Connes amenable dual Banach algebra has
an identity.
\end{theo}

\begin{rem}\label{1.9}
The fact that $\mathcal{V}_{f} $ is non-commutative pointwise Connes
amenable but non-unital, shows that the commutativity in Theorem
\ref{1.11} can not be dropped.
\end{rem}

\begin{prop}\label{1.6} Let $\mathcal{V}$ be a Banach space, and let $f \in
\mathcal{V}^*$ be a non-zero element such that $|| f || \leq 1$.
Then $\mathcal{V}_{f} $ is not approximately amenable.
\end{prop}
\begin{proof} As $\mathcal{V}_{f} $ has no right approximate
identity, it is not approximately amenable by \cite [Lemma 2.2]{GL}.
\end{proof}
The following is a result of F. Ghahramani, see \cite [Theorem
1.5.4]{DL}.
\begin{theo}\label{1.10} Every pointwise amenable, commutative Banach
algebra is  approximately amenable.
\end{theo}
\begin{rem}\label{1.7} By Theorem \ref{1.2} and Proposition
\ref{1.6}, $\mathcal{V}_{f} $ is a pointwise amenable,
non-approximately amenable Banach algebra which is not commutative.
So, without commutativity Theorem \ref{1.10} is not true.
\end{rem}

We conclude by looking at character amenability of $\mathcal{V}_{f}
$. For a Banach algebra $\mathfrak{A}$, we write
$\Delta(\mathfrak{A})$ for the set of all continuous homomorphisms
from $\mathfrak{A}$ onto $ \mathbb{C}$. From \cite{l}, we recall
that a Banach algebra $\mathfrak{A}$ is $\varphi$-\textit{amenable},
$\varphi \in \Delta(\mathfrak{A})$, if there exists a bounded linear
functional $m$ on $\mathfrak{A}^*$ satisfying $ m(\varphi) = 1$ and
$ m(g \cdot a) = \varphi(a) m(g)$, $(a \in \mathfrak{A}, g \in
\mathfrak{A}^*)$. Further, $\mathfrak{A}$ is \textit{character
amenable} if it is $\varphi$-amenable for all $\varphi \in
\Delta(\mathfrak{A})$ \cite{s}.

\begin{prop}\label{1.4} Let $\mathcal{V}$ be a Banach space, and let $f \in
\mathcal{V}^*$ be a non-zero element such that $|| f || \leq 1$.
Then $\mathcal{V}_{f} $ is character amenable.
\end{prop}
\begin{proof} Clearly $f$ is a continuous homomorphisms
from $\mathcal{V}_{f}$ onto $ \mathbb{C}$. Choose $ u \in
\mathcal{V}_{f} $ with $ f(u)= 1$, and set $u_\alpha := u$ for each
$\alpha$. Then $a u_\alpha - f(a) u_\alpha = 0$ for all $a \in
\mathcal{V}_{f} $. Hence, $\mathcal{V}_{f} $ is $f$-amenable by
\cite [Theorem 1.4]{l}. It follows from $\Delta(\mathcal{V}_{f}) =
\{ f \}$ that $\mathcal{V}_{f} $ is character amenable.
\end{proof}

\end{document}